\documentclass[a4paper,draft,reqno,12pt]{amsart}
\usepackage[english]{babel}
\usepackage{amsmath}
\usepackage{amssymb}
\usepackage{amscd}
\usepackage{amsthm}
\usepackage{euscript}
\usepackage{mathtools}
\usepackage[all]{xy}
\usepackage[usenames]{color}
\usepackage{colortbl}

\newtheorem{problem}{Problem}
\newtheorem{lemma}{Lemma}
\newtheorem{theorem}{Theorem}

\theoremstyle{definition}

\newtheorem{example}{Example}
\theoremstyle{remark}
\newtheorem {remark}{Remark}

\def\char{{\rm char\,}}
\def\diag{{\rm diag}}
\def\rk{{\rm rk}}
\def\tr{{\rm tr}}
\def\Ker{{\rm Ker}}
\def\GL{{\rm GL}}
\def\Com{{\rm Com}}
\def\dim{{\rm dim}}

\def\C{\mathbb {C}}
\def\K{\mathbb {K}}
\def\F{\mathbb{F}}
\def\a{\lambda}
\def\sl{\mathfrak{sl}}

\begin{document}
\date{}
\title[On nilpotent generators of the Lie algebra $\mathfrak{sl}_n$]{On nilpotent generators of the Lie algebra $\mathfrak{sl}_n$}
\author{Alisa Chistopolskaya}
\address{Lomonosov Moscow State University, Faculty of Mechanics and Mathematics, Department of Higher Algebra, Leninskie Gory 1, Moscow, 119991 Russia}
\email{achistopolskaya@gmail.com}

\subjclass[2010]{Primary 17B05, 17B22; \ Secondary 15A04}

\keywords{Nilpotent matrix, commutator, Lie algebra, generators}
\maketitle

\begin{abstract}
Consider the special linear Lie algebra $\mathfrak{sl}_n(\K)$ over an infinite field of characteristic different from $2$. We prove that for any nonzero nilpotent $X$ there exists a nilpotent $Y$ such that the matrices $X$ and $Y$ generate the Lie algebra $\mathfrak{sl}_n(\K)$. 
\end{abstract}

\section{introduction}
It is an important problem to find a minimal generating set of a given algebra. This problem was studied actively for semisimple Lie algebras. In 1951, Kuranishi \cite{MK} observed that any semisimple Lie algebra over a field of characteristic zero can be generated by two elements. Twenty-five years later, Ionescu \cite{TI} proved that for any nonzero element $X$ of a complex or real simple Lie algebra $\mathcal{G}$ there exists an element $Y$ such that the elements $X$ and $Y$ generate the Lie algebra $\mathcal{G}$. In the same year, Smith \cite{JS} proved that every traceless matrix of order $n \geqslant 3$ is the commutator of two nilpotent matrices. These results imply that the special linear Lie algebra $\sl_n$ can be generated by three nilpotent matrices. In 2009, Bois \cite{JB} extended Kuranishi's result to algebraically closed fields of characteristic different from $2$ and $3$. 

In this paper we obtain an analogue of Ionescu's result for nilpotent generators of the Lie algebra $\sl_n$. 

\begin{theorem} \label{theorem}
Let $\K$ be an infinite field of characteristic different from $2$. For any nonzero nilpotent $X$ there exists a nilpotent $Y$ such that the matrices $X$ and $Y$ generate the Lie algebra $\mathfrak{sl}_n(\K)$.
\end{theorem}

Our interest to this subject was motivated by the study of additive group actions on affine spaces, see \cite[Theorem~5.17]{AKZ}. In a forthcoming publication we plan to extend Theorem~\ref{theorem} to arbitrary simple Lie algebras.  

The author is grateful to her supervisor Ivan Arzhantsev for posing the problem and permanent support.

\section{main results}

Let $\K$ be an infinite field with $\char\K \neq 2$.
A set of elements $\a_1,\ldots,\a_n$ ($\a_i \in \K $) is called \textit{consistent} if the following conditions hold:
\begin{enumerate}
    \item $\a_1 +\ldots+\a_n = 0$;
    \item $\a_i \neq 0$ for all $i$;
    \item $\a_i \neq \a_j$ for all $i \neq j$;
    \item $\a_i - \a_j = \a_k - \a_l$ only for $i = j,\; k = l$ or $i = k,\; j = l$.
\end{enumerate}

Condition $(1)$ defines $(n-1)$-dimensional subspace $W \subseteq \K^n$. Conditions $(2)$-$(4)$ define linear inequalities on $W$ whose set of solutions is nonempty. For example, if $\char\K = 0$, a set $\a_i = 2^{i-1}$ for $i = 1,\ldots,n-1$ and $\a_n = 1 - 2^{n-1}$ is consistent.

A diagonal matrix $A= \diag(a_{11},\dots,a_{nn})$ is called \textit{consistent} if $a_{11},\dots,a_{nn}$ is a consistent \nolinebreak set.

\begin{lemma} \label{lemma1}
Let $T$ be a consistent matrix and $A$ be a matrix with nonzero entries outside the principal diagonal. Then $T$ and $A$ generate the Lie algebra $\mathfrak{sl}_n(\K)$.
\end{lemma}

\begin{proof}
Consider the following matrices:
$$
A_1 = [T, A],\quad A_i = [T, A_{i-1}],\quad i = 2,\ldots,n^2-n.
$$
Note that all matrices $A_i$ have zeroes on the principal diagonal and are linearly independent. Indeed, if we consider the coordinates of these matrices in the basis of $n^2-n$ matrix units, up to scalar multiplication of columns, we come to a Vandermonde matrix with nonzero determinant. 

Hence, the matrices $A_i$ form a basis of the space of $n \times n$-matrices  with zeroes on the principal diagonal. This means that the subalgebra generated by $T$ and $A$ contains all matrix units $E_{ij}$, $i\neq j$. Since $E_{ii} - E_{jj} = [E_{ij}, E_{ji}]$, it follows that this subalgebra is \nolinebreak $\mathfrak{sl}_n(\K)$.
\end{proof}

\begin{lemma} \label{lemma2}
Let $C$ be a diagonal matrix with pairwise distinct nonzero entries on the principal diagonal. Then $C$ can be represented as $C = A + B$ with $A$ and $B$ being nilpotent matrices, $\rk A =1$, and all entries of A are nonzero.
\end{lemma}

\begin{proof}

Let $V$ be an $n$-dimensional vector space over $\K$. Let us fix a basis $e_1,\dots,e_n$ in $V$ and consider linear operators given by matrices in this basis.

For a non-degenerate matrix $C = \diag(c_{11},\ldots,c_{nn})$, let us consider the following set of vectors: 
$$
v_1 = e_1 +\dots+e_n,\;v_2 = \frac {e_1}{c_{11}} +\dots+\frac{e_n}{c_{nn}},\dots, v_n = \frac {e_1}{c_{11}^{n - 1}} +\dots+\frac{e_n}{c_{nn}^{n-1}}.
$$

Again using a Vandermonde matrix we obtain that $v_1,\dots,v_n$ form a basis in $V$. Moreover, we have $C(v_{i+1}) = v_i$ for all $i=1,\dots,n-1$.
Let us define the operator $B$ as follows:
$B(v_1) = 0$, $B(v_{i+1}) = v_i$ for $i = 1,\dots,n-1$, and let $A = C-B$.

Let us check that $A$ and $B$ satisfy the conditions of Lemma~\ref{lemma2}. Indeed, $B$ is nilpotent and $\rk A =1$, because $A(v_i) = 0$ for $i=1,\dots,n-1$. Since $\tr A =0$, $A$ is also nilpotent.
It is only left to show that in the basis $e_1,\ldots,e_n$ all entries of $A$ are nonzero. All columns of $A$ are proportional to the column $(c_{11},\dots,c_{nn})$, because $\rk A =1$ and $A(e_1 +\dots+e_n) = c_{11}e_1 + \dots +c_{nn}e_n$. Finally, we have $A(e_i) \neq 0$, because the vectors $v_2, v_3, \dots, v_n, e_i$ are linearly independent and $v_2, v_3, \dots, v_n$ is a basis of $\Ker A$.
 \end{proof}
 
\begin{lemma} \label{lemmark1}
For each nilpotent matrix $N$ with $\rk N = 1$ there exists a nilpotent matrix $M$ such that $N$ and $M$ generate $\mathfrak{sl}_n(\K)$. 
\end{lemma}

\begin{proof}
Let $T$ be a consistent matrix, $A$ and $B$ be matrices from Lemma~\ref{lemma2}. It follows from Lemma~\ref{lemma1} that $A$ and $B$ generate $\mathfrak{sl}_n(\K)$. All nilpotent matrices of rank $1$ are conjugate, i.e. for any nilpotent $N$ with $\rk N = 1$ there exists $C \in \GL_n(\K)$ such that $N = CAC^{-1}$. Moreover, if $A$ and $B$ generate $\mathfrak{sl}_n(\K)$ and $C \in \GL_n(\K)$, then $N$ and $CBC^{-1}$ also generate $\mathfrak{sl}_n(\K)$. This completes the proof.
\end{proof}

\begin{lemma}\label{helpZariski}
Let $V$ be a finite-dimensional vector space over an arbitrary
field $\K$. Then a \nolinebreak set of collections consisting of ordered $n$ linearly independent vectors is open in the Zariski topology on $V^n$.

\end{lemma}

\begin{proof}
We fix a basis of $V$. For every set $\{v_1,\dots,v_n\; | \;v_l \in V\}$ let us build a matrix consisting of $n$ rows and $\dim V$ columns such that $l$-th row consists of the coordinates of \nolinebreak $v_l$. It is only left to notice that $v_1,\dots,v_n$ are linearly independent if and only if  there is at least one nonzero minor of order $n$.
\end{proof}

\begin{lemma}\label{Zariski}
For any given matrix $B$ a set of matrices $X$ such that $B$ and $X$ generate $\mathfrak{sl}_n(\K)$ is open in the Zariski topology on $\mathfrak{sl}_n(\K)$.
\end{lemma}

\begin{proof} For any given matrix $B$ and a variable $X$ let us consider a set of all matrices that can be obtained from $B$ and $X$ by means of the Lie bracket [ , ]:
$$
\Com(X, B) \coloneqq \{X,\,B,\; [X,\,B],\,[B,\,X],\,\lbrack X,\,[X,\,B]],\,[B,\,[X,\,B]],\,[[X,\,B],\, X],\dots \}
$$
Let us enumerate all the matrices from $\Com(X,B)$ in some order independent of X. For example, we put $\Com_1(X)= X$, $\Com_2(X)=B$, $\Com_3(X) = [X,\,B]$, $\Com_4(X) = [X,\,[X,\,B]]$, $\Com_5(X) = [B,\,[X,\,B]],\ldots$

Obviously, a subalgebra of $\mathfrak{sl}_n(\K)$ generated by $X$ and $B$ is a linear span of elements of $\Com(X,B)$. Let $I$ be an arbitrary set of indices ${i_1,\ldots,i_{n^2-1}}$ and let $M_I$ be a set of matrices $X$ such that $\Com_{i_{1}}(X),\ldots,\Com_{i_{n^2-1}}(X)$ are linearly independent. Let us construct a map 
$$
\varphi_I: \mathfrak{sl}_n(\K) \longrightarrow (\mathfrak{sl}_n(\K))^{n^2-1} \quad \text{by the rule}\quad X \rightarrow (\Com_{i_{1}}(X),\ldots,\Com_{i_{n^2-1}}(X)).
$$
It is defined by polynomials. Let us look at ordered collections consisting of $n^2-1$ linearly independent matrices. According to Lemma~\ref{helpZariski}, a set compiled from all such collections is open in the Zariski topology on $(\mathfrak{sl}_n(\K))^{n^2-1}$. Then the preimage of this set under $\varphi_I$ is open. Lemma~\ref{Zariski} follows from the fact that a set of matrices $X$ such that $B$ and $X$ generate $\mathfrak{sl}_n(\K)$ is a union of all possible $M_I$.
\end{proof}

\begin{lemma} \label{lemmatheorem}
For any matrix  $A = \sum_{i=1}^{n-1} {a_i E_{i\,i+1}},\; a_1 = 1$, there exists a nilpotent matrix $B$ such that $A$ and $B$ generate $\mathfrak{sl}_n(\K)$.
\end{lemma}

\begin{proof}
Lemma~\ref{lemmark1} implies that there exists a nilpotent matrix $B_0$ such that $E_{12}$ and $B_0$ generate $\mathfrak{sl}_n(\K)$. Consider all matrices of the form
$$
X = \sum_{i=1}^{n-1} {x_i a_i E_{i\,i+1}}.
$$
Obviously, matrices $X$ and $A$ are conjugate if $x_i \neq 0$ for all $i$.
According to Lemma~\ref{Zariski}, there is a polynomial $F(x_1,\dots,x_{n-1})$ such that the following conditions hold:
\begin{enumerate}
    \item $F(1,0,\dots,0) \neq 0;$
    \item $F(x_1,\dots,x_{n-1}) \neq 0 \Longrightarrow B_0$ and $X$ generate $\mathfrak{sl}_n(\K).$
\end{enumerate}
Since $F$ is a nonzero polynomial, there exists a set of nonzero numbers $\a_1,\dots,\a_{n-1}$ such that $F(\a_1,\dots,\a_{n-1}) \neq 0$. It implies that matrices $B_0$ and $X_0 =\sum_{i=1}^{n-1} {\a_i a_i E_{i\,i+1}}$ generate \nolinebreak $\mathfrak{sl}_n(\K)$.  
The fact that $A$ and $X_0$ are conjugate completes the proof.
\end{proof}

\begin{proof}[Proof of Theorem~\ref{theorem}]
For any non-degenerate nilpotent linear operator $X$
there exists a basis such that the matrix of $X$ in this basis has the following form: 
$$
A = \sum_{i=1}^{n-1} {a_i E_{i\,i+1}},\;\; \; a_1 = 1.
$$
In other words, if $X$ is a nonzero nilpotent matrix, there exists $C \in \GL_n(\K)$ such that $A = CXC^{-1}$. It follows from Lemma~\ref{lemmatheorem} that there exists a nilpotent matrix $B$ such that $A$ and $B$ generate $\mathfrak{sl}_n(\K)$. Thus, $X$ and $C^{-1}BC$ generate $\mathfrak{sl}_n(\K)$.
\end{proof}

\section{examples and problems}

Let us give two examples with specific pairs of nilpotent matrices generating $\mathfrak{sl}_n(\K)$. Matrices from the first example generate $\sl_n(\K)$ over an infinite field $\K$ of arbitrary characteristic except for $n = 4$ if $\char\K = 2$.

\begin{example}
Let $\K$ be an infinite field and $\a_1,\dots,\a_n$ ($\a_i \in \K$) be a set such that following conditions hold:
\begin{enumerate}
    \item $\a_1 +\ldots+\a_n = 0$;
    \item $\a_{i+1} \neq \a_i$ for all $i$;
    \item $\a_{i+1} - \a_i = \a_{k+1} - \a_k$ only for $i = k;$
    \item $\a_1+\dots+\a_k \neq 0$ for all $k = 1,\dots,n-1.$
\end{enumerate}
Such sets exist except for $n = 4$ if $\char\K = 2$ (in this case condition (1) implies $\a_{44} - \a_{33} = \a_{22}-\a_{11}$). Let us denote by $s_k$ the element $\a_1+\ldots+\a_k$ and consider the following matrices:
$$A = \begin{pmatrix}
  0 & 1 & 0  & \ldots & 0\\
  0 & 0 & 1  & \dots & 0 \\
  \vdots &  \vdots & \vdots & \ddots & \vdots \\
  0 & 0 & 0 &  \dots & 1\\
  0 & 0 & 0 &  \dots & 0
\end{pmatrix} \quad B = \begin{pmatrix}
  0 &  \ldots & 0 & 0 & 0\\
  s_1 &  \dots & 0 & 0 & 0\\
  \vdots &\ddots & \vdots & \vdots & \vdots \\
  0 & \dots & s_{n-2} & 0 & 0\\
  0 &  \dots& 0  & s_{n-1} & 0
\end{pmatrix}$$ 
Let us show that A and B generate $\mathfrak{sl}_n(\K)$. Indeed, $T = [A, B]$ is a diagonal matrix with the entries $t_{ii} = \a_i$. Similarly to the proof of Lemma~\ref{lemma1}, we obtain that matrices $T$ and $A$ generate a subalgebra of $\mathfrak{sl}_n(K)$ containing \nolinebreak $E_{i\,i+1}$ for $i = 1,\ldots,n-1$. Since $[E_{ik}, E_{kl}] = E_{il}$,
this subalgebra contains all upper nil-triangular matrices.
Similarly, $T$ and $B$ generate a subalgebra containing all lower nil-triangular matrices.
Since $E_{ii} - E_{jj} = [E_{ij}, E_{ji}]$, $A$ and $B$ generate $\mathfrak{sl}_n(\K)$.
\end{example}

\begin{example} 
If $\char \K = 0$ and $n$ is odd, then the matrices 
$$
M = \begin{pmatrix}
  0 & 1 & 0 &  \ldots & 0\\
  0 & 0 & 1 &  \dots & 0 \\
  \vdots  & \vdots & \vdots & \ddots & \vdots \\
  0 & 0 & 0 &  \dots & 1\\
  0 & 0 & 0 &  \dots & 0
\end{pmatrix} \quad N = \begin{pmatrix}
  0 & 0 & 0  & \ldots & 0\\
  0 & 0 & 0  & \dots & 0 \\
  \vdots &   \vdots & \vdots & \ddots & \vdots \\
  0 & 0 & 0 &  \dots & 0\\
  1 & 0 & 0 &  \dots & 0
\end{pmatrix}
$$
generate $\mathfrak{sl}_n(\K)$.
Firstly let us consider the case $\K = \C$. It is possible to make a consistent set from different complex $n$-th roots of unity, since if $n$ is odd, a regular $n$-gon does not have any parallel and equal sides/diagonals. Let $T$ be a corresponding consistent matrix. It can be represented as $T = A+B$, where $A$, $B$ are the nilpotent matrices from Lemma~\ref{lemma2}. Using notations of Lemma~\ref{lemma2}, in the basis $v_1,\ldots,v_n$ the operators $A$ and $B$ have matrices $N = E_{1n}$ and $M = \sum_{i=1}^{n-1} {E_{i\,i+1}}$, respectively. Thus, $M$ and $N$ generate $\mathfrak{sl}_n(\K)$.
We conclude that the set $\Com(M,N)$ from Lemma~\ref{Zariski} contains $n^2-1$ linearly independent matrices. Since linear independence of matrices does not depend on the ground field, the matrices $M$ and $N$ generate $\mathfrak{sl}_n(\mathbb{Q})$ and $\mathfrak{sl}_n(\K)$, where $\K$ is an extension of the field $\mathbb{Q}$.
\end{example}

\begin{remark} If $n$ is even, $M$ and $N$ do not generate $\sl_n(\K)$.

Let us look at the set of matrices 
$$\Lambda = \{A \in \sl_n(\K)\; |\; AC^{-1} + C^{-1}A^T = 0\}, \quad \text{where} \quad 
C = 
\begin{pmatrix}
0 & 0 & \dots & (-1)^n \\
\vdots & \vdots & \ddots & \vdots \\ 
0 & 1 &\dots & 0\\
-1 & 0 &\dots & 0 \\
\end{pmatrix}.
$$

If $n \geqslant 3$ then $\Lambda$ is a proper subalgebra of $\sl_n(\K)$, and if  $n$ is even, we have $M,\,N \in \Lambda$.
\end{remark}

The proof of Theorem~\ref{theorem} implies that for any $n > 1$ there exists a number $N$ such that Theorem~\ref{theorem} holds for $\sl_n(\K)$, $|\K| \geqslant N$. It may be interesting to extend Theorem~\ref{theorem} to finite fields and fields of characteristic $2$.

\begin{problem}
What is the minimal generating set consisting of nilpotent matrices of the Lie algebra $\sl_n$ over a finite field?
\end{problem}

\begin{problem}
Does Theorem~\ref{theorem} hold for infinite field of characteristic $2$? 
\end{problem}

The following example shows that at least some conditions of Theorem~\ref{theorem} are necessary.

\begin {example} 
Let $\F_2$ be the field $\mathbb{Z} / 2\mathbb{Z}$. Let us show that for $\K = \F_2$, Theorem~\ref{theorem} does not hold. We claim that for the matrix $E_{12} \in \sl_3(\F_2)$ there does not exist a nilpotent matrix $Y$ such that $E_{12}$ and $Y$ generate  \nolinebreak $\sl_3(\F_2)$.

Consider linear operators given by the matrices $E_{12}$ and $Y$. If $\rk Y = 1$ then $\Ker Y$ and $\Ker E_{12}$ have nonempty intersection, hence the subalgebra generated by $E_{12}$ and $Y$ is not $\sl_3(\F_2)$. Thus $\rk Y = 2$. Since all nilpotent matrices of rank $2$ are conjugate in $\sl_3(\F_2)$, we only have to check that there is no $A$ of rank $1$ such that $A$ and $B = E_{12} + E_{23}$ generate \nolinebreak $\sl_3(\F_2)$.

Since the first column and the last row of the matrix $B$ are zero, the first column and the last row of the matrix $A$ are nonzero. It implies that $a_{31} = 1$. There are only $8$ such matrices. Two matrices of these eight are persymmetric and we can split other six matrices into pairs symmetric with respect to the antidiagonal matrices. Let $X'$ be a matrix such that $X$ and $X'$ are symmetric with respect to the antidiagonal. Since symmetry and antisymmetry are the same in $\F_2$ and $B = B'$, we have $[X, B]' = [X', B]$. It implies that if $X = X'$  the subalgebra $\sl_3(\F_2)$ generated by $X$ and $B$ consists of matrices symmetric with respect to the antidiagonal. Moreover, if $X$ and $B$ generate $\sl_3(\F_2)$, then $X'$ and $B'$ generate $\sl_3(\F_2)$. So it is only left to show that $A$ and $B$ do not generate $\sl_3(\F_2)$, where $A$ is one of the following three matrices:
$$
A_1 = \begin{pmatrix}
1 & 1 & 0 \\
1 & 1 & 0 \\
1 & 1 & 0\\
\end{pmatrix} ,  \quad 
A_2 = \begin{pmatrix}
1 & 0 & 1 \\
1 & 0 & 1 \\
1 & 0 & 1\\
\end{pmatrix} , \quad
A_3 = \begin{pmatrix}
0 & 0 & 0 \\
1 & 0 & 0 \\
1 & 0 & 0\\
\end{pmatrix} 
$$ 

Let us denote by $\Lambda_1$ the linear span of the matrices 
$$
\begin{pmatrix}
1 & 1 & 0 \\
1 & 1 & 0 \\
1 & 1 & 0\\
\end{pmatrix}  \quad 
\begin{pmatrix}
0 & 1 & 0 \\
0 & 0 & 1 \\
0 & 0 & 0\\
\end{pmatrix} \quad
\begin{pmatrix}
1 & 1 & 1 \\
0 & 0 & 1 \\
0 & 0 & 1\\
\end{pmatrix} \quad
\begin{pmatrix}
1 & 0 & 1 \\
1 & 0 & 1 \\
0 & 1 & 1\\
\end{pmatrix} 
$$

and by $\Lambda_2$ the linear span of the matrices 

$$
\begin{pmatrix}
1 & 0 & 1 \\
1 & 0 & 1 \\
1 & 0 & 1\\
\end{pmatrix}  \quad 
\begin{pmatrix}
0 & 1 & 0 \\
0 & 0 & 1 \\
0 & 0 & 0\\
\end{pmatrix} \quad
\begin{pmatrix}
  1 & 1 & 1 \\
  1 & 1 & 1 \\
  0 & 1 & 0\\
\end{pmatrix} \quad
\begin{pmatrix}
1 & 0 & 0 \\
0 & 0 & 1 \\
0 & 0 & 1\\
\end{pmatrix} \quad
\begin{pmatrix}
0 & 1 & 1 \\
0 & 0 & 1 \\
0 & 0 & 0\\
\end{pmatrix} 
$$
We have $A_1,\, B \in \Lambda_1$ and $A_2,\,A_3,\,B \in \Lambda_2$, and it is easy to verify directly that $\Lambda_1$ and $\Lambda_2$ are subalgebras of $\sl_n(\K)$. 
\end{example}

\end{document}